\documentclass[11pt]{article}
\usepackage{a4wide}
\usepackage{amsmath, amsthm, xcolor}
\usepackage{amssymb}
\usepackage{graphicx,soul}
\usepackage[normalem]{ulem}
\usepackage{relsize}
\usepackage{enumerate}
\newcommand{\hidden}[1]{}
\usepackage[colorlinks, linkcolor=blue, anchorcolor=blue, citecolor=blue]{hyperref}

\usepackage{color}
\usepackage[mathscr]{euscript}
\usepackage{tikz}
\numberwithin{equation}{section}

\newtheorem{theorem}{Theorem}[section]
\newtheorem{proposition}[theorem]{Proposition}

\newtheorem{corollary}[theorem]{Corollary}
\newtheorem{lemma}[theorem]{Lemma}

\theoremstyle{remark}
\newtheorem{remark}[theorem]{{Remark}}

\begin{document}

\title{Rational maps associated with the Sigmoid Beverton-Holt model on the projective line over $\mathbb{Q}_p$}


 \author{Cheng Liu \\ {\small\sc (Wuhan) }}

\bigskip

\date{}

 \maketitle

\begin{abstract}
	We describe the dynamical structure of the $p$-adic rational dynamical systems associated with the Sigmoid Beverton-Holt model on the projective line $\mathbb{P}^1(\mathbb{Q}_p)$ over the field $\mathbb{Q}_p$ of $p$-adic numbers. Our methods are minimal decomposition of $p$-adic polynomials with coefficients in $\mathbb{Z}_p$ established by Fan and Liao and the chaotic description of $p$-adic repellers of Fan, Liao, Wang and Zhou.
\end{abstract}
$\textbf{Keywords}$: $p$-adic dynamical system,
Sigmoid Beverton-Holt model, minimal decomposition, $p$-adic repeller

\bigskip

\section{Introduction}
Let $p\ge 2$ be a prime. Let $\mathbb{Q}_p$ be the field of $p$-adic numbers and $\mathbb{P}^1(\mathbb{Q}_p)$ be the projective line over $\mathbb{Q}_p$. Polynomial maps with coefficients in $\mathbb{Q}_p$, have been extensively studied as natural examples of dynamical systems on $\mathbb{Q}_p$. In \cite{fl11}, Fan and Liao pointed out that the dynamical structure of polynomials with coefficients in the ring $\mathbb{Z}_p$ of $p$-adic integers is often determined by their minimal subsystems. They also provided a complete description of the dynamical structure of quadratic polynomial dynamical systems on $\mathbb{Z}_2$. Since then, a significant amount of work has been done on special polynomial dynamical systems using the method from \cite{fl11}. Some examples include the study of square mappings on $\mathbb{Z}_p$ in \cite{fl16s} and Chebyshev polynomials on $\mathbb{Z}_2$ in \cite{fl16c}. Additionally, there has been considerable interest in $p$-adic chaos. In \cite{ws98}, Woodcock and Smart demonstrated that the polynomial $\frac{x^p-x}{p}$ on $\mathbb{Z}_p$ is topologically conjugate to the full shift on $\{0,\cdots,p-1\}^{\mathbb{N}}$, hence, exibits some chaotic property of the system. This chaotic property has also been investigated for many other polynomials in \cite{dsv06,flw07,tvw89}. In particular, in \cite{flw07}, Fan, Liao, Wang and Zhou proved that transitive $p$-adic repellers (including a large family of polynomials) are topologically conjugate to some subshifts of finite type.
\par
As a natural generalization of polynomial maps, rational maps have also been widely studied. For rational maps of degree $1$ on $\mathbb{P}^1(\mathbb{Q}_p)$, Fan, Fan, Liao and Wang \cite{fflw14} completely depicted their dynamical structure. Then, in \cite{fflw17}, the same authors proved that rational maps of good reduction with degree at least $2$ have a similar dynamical structure. Since the dynamical structure of general rational maps is quite difficult to describe, people focus on depicting the dynamical structure for some special rational maps. For example, in \cite{fl18,sfl18}, Fan and Liao fully described the dynamical structure of $ax+1/x$ on $\mathbb{P}^1(\mathbb{Q}_p)$. The chaotic property of some other rational maps was studied in \cite{als18,mk16,mk18,mk22}.
\par
In this paper, we are interested in studying the dynamical structure of rational maps associated with the Sigmoid Beverton-Holt model. The Sigmoid Beverton-Holt model has a wide range of practical applications, such as predicting insurance premiums, forecasting natural populations and determining appropriate fishing rates. We refer to \cite{gkfcr12,hkk12} for more details.\par
Let $N\ge 1$ be an integer and let $a\in \mathbb{Q}_p\setminus \{0\}$ be a nonzero $p$-adic number. Consider the $p$-adic rational
dynamical system $\left(\mathbb{P}^1(\mathbb{Q}_p),\phi_N\right)$, where 
\begin{align}\label{perf}
  \phi_N(z)=\frac{az^N}{z^N+1}
\end{align}
is the recruitment function of the Sigmoid Beverton-Holt model. In \cite{dt24}, Diagana studied the fixed points and maximal Siegel disks of \eqref{perf} under the assumptions $gcd(N,p)=1$ and $0<|a|_p<1$. Our main goal is to describe the long-term behavior of the system $\left(\mathbb{P}^1(\mathbb{Q}_p),\phi_N\right)$ for all $a\in \mathbb{Q}_p\setminus\{0\}$ without assuming $gcd(N,p)=1$.
\par
For convenience, we denote by
\[D(x_0,r):=\{x\in \mathbb{Q}_p:|x-x_0|_p\le r\}\]
the closed which is also open disk centered at $x_0$ with radius $r$, and 
\[S(x_0,r):=\{x\in \mathbb{Q}_p:|x-x_0|_p= r\}\]
the sphere centered at $x_0$ with radius $r$.
The following theorem reduces the study of our rational dynamical system to the study of a polynomial dynamical system.
\begin{proposition}\label{thm2.5}
  Let $F(x)=ax^n\in\mathbb{Q}_p[x]$ with $n\ge 1$. Suppoes $G(x)=b^mx^m+\cdots+b_0\in\mathbb{Q}_p[x]$ with $b_mb_0\neq 0$.
  Let $\phi(x)=\frac{F(x)}{G(x)}$. If $m\le n$, then $\left(\mathbb{P}^1(\mathbb{Q}_p),\phi\right)$ is topologically conjugate to $\left(\mathbb{P}^1(\mathbb{Q}_p),H\right)$ with 
  \[H(x)=\frac{x^nG(x^{-1})}{a}\in\mathbb{Q}_p[x]\] 
  and $deg(H)=n$.
\end{proposition}
\begin{proof}
  The theorem follows by choosing the function $\pi: x\mapsto \frac{1}{x}$ as a conjugacy.
\end{proof}
By Proposition \ref{thm2.5}, we have that $\left(\mathbb{P}^1(\mathbb{Q}_p),\phi_N\right)$ is topologically conjugate to $\left(\mathbb{P}^1(\mathbb{Q}_p),f_N\right)$, where
\[f_N(x)=\frac{x^N+1}{a}.\]
Since the case $N=1$ has been studied by Fan, Li, Yao and Zhou in \cite{flyz07}, in what follows, we will study the dynamical system $(\mathbb{P}^1(\mathbb{Q}_p),f_N)$ with $N\ge 2$. \par
When $|a|_p>1$, we have the following theorem.
\begin{theorem}\label{sy1}
  Suppose $|a|_p>1$. Then, we have 
  \[\lim_{n\to\infty}f_N^n(x)=\infty,\quad \forall x\in \{\infty\}\cup\mathbb{Q}_p\setminus D\left(0,|a|_p^{\frac{1}{N-1}}\right)\]
  and
  \[\lim_{n\to\infty}f_N^n(x)=x_0,\quad \forall x\in D\left(0,\ \frac{1}{p}|a|_p^{\frac{1}{N-1}}\right)\]
  where $x_0$ is the fixed point of $f_N$ with $|x_0|_p=\frac{1}{|a|_p}$. Moreover, if $(N-1)\nmid v_p(a)$, then $S(0,|a|_p^{\frac{1}{N-1}})=\emptyset$ and $\left(\mathbb{P}^1(\mathbb{Q}_p),f_N\right)$ has exactly two fixed points $\infty$ and $x_0$. Otherwise, the subsystem $\left(S(0,|a|_p^{\frac{1}{N-1}}), f_N\right)$ is topologically conjugate to $(S(0,1), g_N)$, where $g_N\in\mathbb{Z}_p[x]$ is defined by
  \begin{align}\label{dyn1}
    g_N(x)=\frac{1}{a|a|_p}x^N+\frac{1}{a}|a|_p^{\frac{1}{N-1}}.
  \end{align}  
\end{theorem}
\begin{remark}
  The dynamical structure of $(S(0,1), g_N)$ can be described by minimal decomposition. We will show more about this in Theorem \ref{dsy1}.
\end{remark}
When $|a|_p<1$, set $\mathcal{A}_{f_N}:=\left\{x\in\mathbb{Q}_p:|f_N^n(x)|_p=1,\ \forall n\ge 0\right\}$. Let $\sigma$ be the left shift map on $\Sigma_{\ell}:= \{0,\cdots,\ell-1\}^{\mathbb{N}}$. We have the following theorems.
\begin{theorem}\label{sy2}
  Let $p\ge 3$ and let $|a|_p<1$. Let $N=qp^m$ with $gcd(q,p)=1$.
  \begin{itemize}
   \item[(i)] Suppose $gcd(p-1,q)\mid \frac{p-1}{2}$ and $|a|_p<|N|_p$. If $f_N$ has $\ell\le 2$ different fixed points $\{\omega_{i,N}\}_{i=1}^{\ell}$ with $\omega_{i,N}\not\equiv \omega_{j,N}\ ({\rm mod}\ p)$, then $(\mathcal{K}_{f_N}, f_N)$ is topologically conjugate to $(\Sigma_{\ell}, \sigma)$, where 
    \[\mathcal{K}_{f_N}:=\bigcap_{n=0}^{\infty}f_N^{-n}(X)\ \text{with}\ X=\bigsqcup_{i=1}^{\ell} D(\omega_{i,N},\frac{1}{p}).\]
    Moreover, if $|a|_p<|N|_p^2$ and $gcd(p-1,q)\ge 2$, then we have $\ell=gcd(p-1,q)$ and $\mathcal{K}_{f_N}=\mathcal{A}_{f_N}$.
    \item[(ii)] Suppose $gcd(p-1,q)\nmid \frac{p-1}{2}$ or $|a|_p\ge |N|_p$. Then, we have 
    \[\lim_{n\to\infty}f_N^n(x)=\infty,\quad \forall x\in\mathbb{Q}_p\cup\{\infty\}.\]
  \end{itemize}
\end{theorem}
\begin{theorem}\label{sy3}
  If $|a|_2<1$, then the dynamical system $(\mathbb{P}^1(\mathbb{Q}_2), f_N)$ is described as follows.
  \begin{itemize}
    \item[(i)] If $|a|_2=2^{-k}$ with $k\ge 1$ and $2\nmid N$, then
    $(\mathbb{P}^1(\mathbb{Q}_2), f_N)$ has exactly two fixed points $\infty$ and $x_0\in D(2^k-1,2^{-k-1})$. Moreover, 
    \[\lim_{n\to\infty}f_N^n(x)=\infty,\quad \forall x\neq x_0.\] 
    \item[(ii)] If $|a|_2\le \frac{1}{4}$ and $2\mid N$, then we have 
    \[\lim_{n\to\infty}f_N^n(x)=\infty,\quad \forall x\in \mathbb{Q}_2\cup\{\infty\}.\]
    \item[(iii)] If $|a|_2= \frac{1}{2}$ and $|N|_2\le \frac{1}{4}$, then $f_N$ has only one fixed point $x_0\in S(0,1)$. Moreover, 
    \[\lim_{n\to\infty}f_N^n(x)=
      \begin{cases}
      x_0 & \text{if}\ x\in S(0,1),\\
      \infty & \text{if}\ x\notin S(0,1).
      \end{cases}
    \]
    \item[(iv)] If $|a|_2=\frac{1}{2}$ and $|N|_2=\frac{1}{2}$, then we have
    \[\lim_{n\to\infty}f_N^n(x)=\infty,\quad\forall x\not\in S(0,1)\]
    and the subsystem $\left(S(0,1), f_N\right)$ is topologically conjugate to $(\mathbb{Z}_p, h_N)$ where $h_N\in\mathbb{Z}_p[x]$ is defined by
      \begin{align}\label{hn}
        h_N(x)=\frac{(2x+1)^N+1-a}{2a}.
      \end{align}
  \end{itemize}
\end{theorem}
When $|a|_p=1$, we have the following theorem.
\begin{theorem}\label{sy4}
  If $|a|_p=1$, then we have 
  \[\lim_{n\to\infty}f_N^n(x)=\infty,\quad \forall x\not\in \mathbb{Z}_p.\]
\end{theorem}
Now, we focus on three $p$-adic polynomial dynamical systems: the system $(S(0,1),g_N)$ in Theorem \ref{sy1}, the system $(\mathbb{Z}_p,h_N)$ in Theorem \ref{sy3} and the system $(\mathbb{Z}_p,f_N)$ in Theorem \ref{sy4}. These dynamical systems are often described by their minimal subsystems. Considering a dynamical system $(X,T)$. A subset $E\subset X$ is \textit{invariant} if $T(E)\subset E$. Then, we have a subsystem $(E,T|_E)$. The subsystem $(E,T|_E)$ is \textit{minimal} if the orbit of any point in $E$ is dense in $E$. In this case, we also say that $E$ is a \textit{minimal set}.   
\par
We underline that the behavior of $(S(0,1),g_N),\ (\mathbb{Z}_p,h_N)$ and $(\mathbb{Z}_p,f_N)$ highly depends on parameters $N$ and $a$. 
So, instead of considering the whole system for all $N$ and $a$, we focus on some special cases.
\par
For the system $(S(0,1),g_N)$, we provide a necessary and sufficient condition under which the whole system is minimal.
\begin{theorem}\label{dsy1}
  Suppose $|a|_p>1$ and $(N-1)\mid v_p(a)$. Define $g_N$ as \eqref{dyn1}.
  \begin{enumerate}[(i)]
    \item If $p=2$, then the whole dynamical system $(S(0,1),g_N)$ can never be minimal.
    \item If $p=3$, then the whole dynamical system $(S(0,1),g_N)$ is minimal if and only if both of the following two conditions hold:
    \begin{enumerate}[1)]
      \item $N\equiv 1\ ({\rm mod}\ 2)$ and $|a|_3\cdot a\equiv 2\ ({\rm mod}\ 3)$.
      \item $\left(g_N^2\right)^{\prime}(1)\equiv 1\ ({\rm mod}\ 3)$, 
      \[\frac{g^2_N(1)-1}{3}\not\equiv 0 \ ({\rm mod}\ 3),\quad \frac{g^2_N(1)-1}{3}\not\equiv \frac{\left(g_N^2\right)^{\prime\prime}(1)}{2}\ ({\rm mod}\ 3).\]
    \end{enumerate}
    \item If $p\ge 5$, then the whole dynamical system $(S(0,1),g_N)$ is minimal if and only if both of the following two conditions hold:
    \begin{enumerate}[1)]
      \item Letting $m\in\mathbb{N}$ be such that $|a|_p\cdot ag^{m}\equiv 1\ ({\rm mod}\ p)$, where $g$ denotes a generator of $\mathbb{F}_p^*$, we have
      \begin{align}\label{bcy}
        \min\left\{n\ge 1:\frac{N^n-1}{N-1}m\equiv 0\ ({\rm mod}\ p-1)\right\}=p-1.
      \end{align}
      \item $\left(g_N^{p-1}\right)^{\prime}(1)\equiv 1\ ({\rm mod}\ p)$ and
      \[\frac{g^{p-1}_N(1)-1}{p}\not\equiv 0\ ({\rm mod}\ p).\]
    \end{enumerate}
  \end{enumerate}
\end{theorem}
For the system $(\mathbb{Z}_p,h_N)$, the behavior could be very complex. We only point out that the case $N=2$ and $p=2$ is fully described by Fan and Liao in \cite{fl11}.\par
For the system $(\mathbb{Z}_p,f_N)$ with $|a|_p=1$, we only consider the cases $p=2$ and $p=3$.\par
When $p=2$, we can fully describe $(\mathbb{Z}_2,f_N)$ with $|a|_2=1$.
\begin{theorem}\label{dsy2}
  Suppose $|a|_2=1$. Then, the function $f_N$ admits a $2$-periodic point $x_0$ in $\mathbb{Z}_p$ and $\mathbb{Z}_p$ is contained in the attracting basin of the periodic orbit $\left(x_0,\ f_N(x_0)\right)$.
\end{theorem}
Now, we will show that even we suppose $p=3$ and $a\equiv 1\ ({\rm mod}\ 3)$, the dynamical behavior of the system $(\mathbb{Z}_p,f_N)$ is still very complex.
\begin{theorem}\label{dsy3}
  Suppose $p=3$ and $a\equiv 1\ ({\rm mod}\ 3)$. Then, the system $(\mathbb{Z}_3,f_N)$ can be partially described as follows.
  \begin{enumerate}[(i)]
    \item If $a\equiv 1\ ({\rm mod}\ 3)$ and $2\nmid N$, then $f_N$ admits a $3$-periodic point $x_0$ in $\mathbb{Z}_3$ and $\mathbb{Z}_3$ is contained in the attracting basin of the periodic orbit $\left(x_0,\ f_N(x_0),\ f^2_N(x_0)\right)$.
    \item If $a\equiv 1\ ({\rm mod}\ 3)$ and $2\mid N$, then $f_N(3\mathbb{Z}_3)\subset f_N(1+3\mathbb{Z}_3),\ f_N(1+3\mathbb{Z}_3)\subset f_N(2+3\mathbb{Z}_3)$ and $2+3\mathbb{Z}_3$ is an $f_N$-invariant set.
    \begin{enumerate}[1)]
      \item If $6\mid N$, then $f_N$ admits a fixed point $x_0$ in $2+3\mathbb{Z}_3$ and $2+3\mathbb{Z}_3$ is contained in the attracting basin of $x_0$.
      \item The compact open set $(2+9\mathbb{Z}_3)\bigsqcup (8+9\mathbb{Z}_3)$ is minimal if and only if $a\equiv 1\ ({\rm mod}\ 9)$ and $N\equiv 4,16\ ({\rm mod}\ 18)$.
      \item The compact open set $(2+3\mathbb{Z}_3)$ is minimal if and only if $a\equiv 4\ ({\rm mod}\ 9)$ and $N\equiv 2\ ({\rm mod}\ 6)$.
      \item The compact open set $(5+9\mathbb{Z}_3)\bigsqcup (8+9\mathbb{Z}_3)$ is minimal if and only if $a\equiv 4\ ({\rm mod}\ 9)$ and $N\equiv 10,16\ ({\rm mod}\ 18)$.
      \item The compact open set $(2+9\mathbb{Z}_3)\bigsqcup (5+9\mathbb{Z}_3)$ is minimal if and only if $a\equiv 7\ ({\rm mod}\ 9)$ and $N\equiv 4,10\ ({\rm mod}\ 18)$.
      \item The compact open sets $(2+9\mathbb{Z}_3),\ (5+9\mathbb{Z}_3)$ and $(8+9\mathbb{Z}_3)$ are minimal if and only if $a\equiv 7\ ({\rm mod}\ 27)$ and $N\equiv 2\ ({\rm mod}\ 18)$ or $a\equiv 16\ ({\rm mod}\ 27)$ and $N\equiv 8,14\ ({\rm mod}\ 18)$.
    \end{enumerate}
  \end{enumerate}
\end{theorem}
In order to overcome obstacles that Diagana encountered in \cite{dt24}, we reduce the study of rational dynamical system to the study of polynomial dynamical system, so that we can use minimal decomposition. We also prove an estimate for binomial coefficients in Lamme \ref{lem2.4} which allow us to use the methods in \cite{flw07}. Finally, we point out that the main results of Diagana in \cite{dt24} are containing in Theorems \ref{sy2} and \ref{sy3}.\par 
The paper is organized as follows.
\section{Preliminaries}
\subsection{\textit{P-adic numbers and the projective line} $\mathbb{P}^1(\mathbb{Q}_p)$}
Let $p\ge 2$ be a prime. We can write any nonzero rational number $r\in \mathbb{Q}$ as $r=p^{v}\frac{a}{b}$ where $v,a,b\in \mathbb{Z}$ and $gcd(p,a)=gcd(p,b)=1$. Define $v_p(r)=v$ and $|r|_p=p^{-v_p(r)}$ for $r\neq 0$ and $|0|_p=0$. Then $|\cdot|_p$ is a non-Archimedean absolute value on $\mathbb{Q}$. The completion of $\mathbb{Q}$ under the absolute value $|\cdot|_p$ is the field $\mathbb{Q}_p$ \textit{of p-adic numbers}. It is known that, any $x\in \mathbb{Q}_p$ has a unique expansion as
\[x=\sum_{n=v_p(x)}^{\infty}x_np^n,\quad x_n\in \{0,1,\cdots,p-1\}\quad \text{and}\quad x_{v_p(x)}\neq 0.\]
Here, the integer $v_p(x)$ is called the \textit{p-valuation} of $x$.\par
We can describe any point in the \textit{projective line} $\mathbb{P}^1(\mathbb{Q}_p)$ by a pair $[x:y]$ of points in $\mathbb{Q}_p$ which are not both zero. Two such pairs are equal if there exists a $\lambda\in \mathbb{Q}_p^*$ such that
\[[x:y]=[\lambda x:\lambda y].\]
By identifying the field $\mathbb{Q}_p$ with the subset of $\mathbb{P}^1(\mathbb{Q}_p)$ given by $\{[x:1]\in \mathbb{P}^1(\mathbb{Q}_p)| x\in \mathbb{Q}_p\}$, we can view $\mathbb{P}^1(\mathbb{Q}_p)$ as $\mathbb{Q}_p\cup \{\infty\}$.\par
The \textit{spherical metric} defined on $\mathbb{P}^1(\mathbb{Q}_p)$ is an analogue to the standard spherical metric on the Riemann sphere. Suppose $P=[x_1:y_1] \ \text{and}\ Q=[x_2:y_2] \in \mathbb{P}^1(\mathbb{Q}_p)$. We define
\[\rho(P,Q)=\frac{|x_1y_2-x_2y_1|_p}{\max\{|x_1|_p,|y_1|_p\}\max\{|x_2|_p,|y_2|_p\}}.\]
Viewing $\mathbb{P}^1(\mathbb{Q}_p)$ as $\mathbb{Q}_p\cup \{\infty\}$, the spherical distance of $x,y\in \mathbb{Q}_p\cup \{\infty\}$ can also be described by
\[\rho(x,y)=\frac{|x-y|_p}{\max\{|x|_p,1\}\max\{|y|_p,1\}}\quad \text{if}\ x,y\in \mathbb{Q}_p,\]
and
\begin{align*}
  \rho(x,\infty)=\begin{cases}
    1, & \text{if}\ |x|_p\le 1,\\
    \frac{1}{|x|_p}, & \text{if}\ |x|_p> 1.
  \end{cases}
\end{align*}
\par
Now, we list some fundamential results about $\mathbb{Q}_p$.
\begin{lemma}\cite[Proposition 2.1]{j11}\label{lem2.2}
  Let $K$ be a field which is complete with respect to a non-trivial, non-Archimedean absolute value $|\cdot |$, and suppose $\mathcal{O}=\{x\in K: |x|\le 1\}$. Let $F\in \mathcal{O}[X]$. Suppose that at some point $x_0\in\mathcal{O}$, there exists $L\ge 1$, so that 
  \begin{align}\label{hl1}
    |F(x_0)|<\left|\frac{F^{(L)}(x_0)}{L!}F^{\prime}(x_0)\right|,
  \end{align}
  and for each $1\le k\le L-1$ $(\text{if}\ L=1,\text{we only need \eqref{hl1}})$,
  \begin{align*}
    \left|\frac{F^{(k+1)}(x_0)}{(k+1)!}F(x_0)\right|< \left|\frac{F^{(k)}(x_0)}{k!}F^{\prime}(x_0)\right|.
  \end{align*} 
  Then, there is a unique $x\in \mathcal{O}$ with $F(x)=0$ and $|x-x_0|\le \left|\frac{F(x_0)}{F^{\prime}(x_0)}\right|$.
\end{lemma}
\begin{lemma}\cite[Proposition 3.1]{dt24}\label{lem2.3}
  Suppose $N=qp^m$ with $gcd(p,q)=1$ and $m\in \mathbb{N}$. Consider the equation
  \begin{align}\label{eq1}
    x^N+1\equiv 0\quad ({\rm mod}\ p).
  \end{align}
  If $p=2$, then the equation \eqref{eq1} has only one solution.
  If $p\ge 3$, then the equation \eqref{eq1} has solutions if and only if $gcd(p-1,q)\mid \frac{p-1}{2}$, and the number of solutions is $gcd(p-1,q)$.
\end{lemma}
The following estimates for binomial coefficients $C_{n}^k$ will be useful in our calculations. For any $n=\sum_{i=0}^{m}a_ip^i$ with $0\le a_i\le p-1$, define 
\[wt_{p}(n)=\sum_{i=0}^{m}a_i.\] 
\begin{lemma}\cite[Corollary 3.7]{ak09}\label{vobc}
  For all $N\ge K\ge 0$, we have
  \begin{align}
    v_p(C_N^K)=\frac{1}{p-1}(wt_{p}(K)+wt_{p}(N-K)-wt_{p}(N)).
  \end{align}
\end{lemma}
\begin{lemma}\label{lem2.4}
  Let $N\ge 2$ be an integer. Then, for any non-negative integer $K\le N-2$, we have 
  \begin{align}\label{zuhe}
    |C_N^K|_p\cdot p^{1-(N-K)}\le |N|_p.
  \end{align}
  The equality holds if and only if $p=2$ and $K=N-2$.
\end{lemma}
\begin{proof}
  Since $N\ge 2$ and $K\le N-2$, we can write $N=p^m(N_0+N_1p+\cdots +N_sp^s)$ and $K=K_0+K_1p+\cdots +K_{m+s}p^{m+s}$ for some $m,\ s\in\mathbb{N}$, where $N_0,\ N_s\neq 0$ and $0\le N_i,\ K_j\le p-1$ for all $0\le i\le s$ and $0\le j\le m+s$.\par
  If $K_{m+s}=N_s$, then by Lemma \ref{vobc}, we have that $v_p(C_N^K)$ is independent of $N_s$. If $m=0$, then $|N|_p=1$, which implies \eqref{zuhe}. Hence, without loss of generality, we may assume $K_{m+s}\le N_s-1$ and $m\ge 1$. Then, we have 
  \begin{align*}
    N-K =(p-K_0)+(p-1-K_1)p+\cdots +(p-1-K_{m-1})p^{m-1} +\\(N_0-1-K_m)p^m+(N_1-K_{m+1})p^{m+1}+\cdots+(N_s-K_{m+s})p^{m+s}:= \sum_{i=0}^{m+s}c_ip^i.
  \end{align*}
  We distinguish four cases.
  \begin{enumerate}[(i)]
    \item If $K_0>0$, then we have
    $1\le p-K_0\le p-1$.
    We distinguish two cases.
    \begin{enumerate}[1)]
      \item If $c_i\ge 0$ for all $0\le i\le m+s$, then we have 
      \[wt_p(N-K)=\sum_{i=0}^{m+s}c_i=m(p-1)+\sum_{i=0}^sN_i-\sum_{j=0}^{m+s}K_j=m(p-1)+wt_p(N)-wt_p(K).\]
      \item If there exists $i$ such that $c_i<0$, then we set 
      \[t=\min\{i: c_i<0\}\ \text{and}\ \ell=\#\{t\le i: c_i\le 0 \}.\]
      We claim that
      \begin{align}\label{wtpe}
        wt_p(N-K)=\sum_{i=0}^{m+s}c_i+\ell (p-1).
      \end{align}
      We prove \eqref{wtpe} by induction on $\ell$. Suppose $\ell=1$. Since $c_{t+1}\ge 1$, we have
      \[N-K=\sum_{i=0}^{t-1}c_ip^i+(p+c_t)p^t+(c_{t+1}-1)p^{t+1}+\sum_{i=t+2}^{m+s}c_ip^i.\]
      Hence,
      \[wt_p(N-K)=\sum_{i=0}^{m+s}c_i+p-1.\]
      Suppose \eqref{wtpe} holds for $\ell\le n$. We then prove \eqref{wtpe} holds for $\ell= n+1$. Suppose $c_{t_j}\le 0$ for all $1\le j\le n+1$ and $t_j$ strictly increasing. If $t_1+1\neq t_2$, then we have
      \[N-K=\sum_{i=0}^{t_1-1}c_ip^i+(p+c_{t_1})p^t+(c_{t_1+1}-1)p^{t_1+1}+\sum_{i=t_1+2}^{m+s}c_ip^i.\]
      Let $d_i=c_i$ for all $i\not\in \{t_1, t_1+1\}$. Let $d_{t_1}=p+c_{t_1}$ and $d_{t_1+1}=c_{t_1+1}-1$. Then we have 
      \[N-K=\sum_{i=0}^{m+s}d_ip^i.\]
      Let $t^{\prime}=\min\{i: d_i<0\}\ \text{and}\ \ell^{\prime}=\#\{t\le i: d_i\le 0 \}$. So by induction
      we have
      \[wt_p(N-K)=\sum_{i=0}^{m+s}d_i+n(p-1)=\sum_{i=0}^{m+s}c_i+(n+1)(p-1).\]
      If $t_1+1=t_2$ and $t_2+1\neq t_3$, then we have
      \[N-K=\sum_{i=0}^{t_1-1}c_ip^i+(p+c_{t_1})p^t+(c_{t_2}-1+p)p^{t_2}+(c_{t_2+1}-1)p^{t_2+1}+\sum_{i=t_2+2}^{m+s}c_ip^i.\]
      By the same arguments above, we have
      \[wt_p(N-K)=\sum_{i=0}^{m+s}c_i+(n+1)(p-1).\]
      Repeating the process, we complete the proof of \eqref{wtpe}.
    \end{enumerate}
    Hence, we have 
    \[wt_p(N-K)+wt_p(K)-wt_p(N)\ge m(p-1).\]
    Hence, by Lemma \ref{vobc}, we have 
    \[|C_N^K|_p\cdot p^{1-(N-K)}\le p^{-m+1-(N-K)}<p^{-m}.\]
    \item If $K_0=0$ and $m=1$, then we have
    \[K=K_1p+\cdots+K_{s+1}p^{s+1}\le (p-1)(p+p^2+\cdots+p^s)+K_{s+1}p^{s+1}=p^{s+1}-p+K_{s+1}p^{s+1}.\]
    So,
    \[N-K\ge p(N_0+\cdots+N_sp^s)-p^{s+1}+p-K_{s+1}p^{s+1}\ge (N_s-K_{s+1}-1)p^{s+1}+p\ge p.\]
    Hence,
    \[|C_N^K|_p\cdot p^{1-(N-K)}\le p^{1-(N-K)}\le p^{1-p}\le p^{-1}.\]
    \item If $K_0=0,\ m\ge 2$ and $K_1\neq 0$, then by the same arguments in (i), we have 
    \[wt_p(N-K)+wt_p(K)-wt_p(N)\ge (m-1)(p-1).\]
    Hence, by Lemma \ref{vobc}, we have 
    \[|C_N^K|_p\cdot p^{1-(N-K)}\le p^{-m+2-(N-K)}\le p^{-m+2-p}\le p^{-m}.\]
    \item If $K_0=K_1=0$ and $m=2$, then we have
    \[N-K\ge (N_{s+2}-K_{s+2}-1)p^{2+s}+p^2=p^2.\]
    Hence,
    \[|C_N^K|_p\cdot p^{1-(N-K)}\le p^{1-(N-K)}\le p^{1-p^2}< p^{-2}.\]
  \end{enumerate}
  Repeating the process, we complete the proof of \eqref{zuhe}. By examining the proof above, we find that the equality holds if and only if $p=2$ and $K=N-2$.
\end{proof}
\begin{corollary}\label{cor2.5}
  Let $p\ge 3$ be a prime. Let $\omega\in\mathbb{Q}_p$. 
  If $r< |\omega|_p$, then for all $x,y\in D(\omega,r)$, we have 
  \[|f_N(x)-f_N(y)|_p=\frac{|N|_p}{|a|_p}|\omega|_p^{N-1}|x-y|_p.\] 
\end{corollary}
\begin{proof}
  Since 
  \[|f_N(x)-f_N(y)|_p=\frac{1}{|a|_p}\left|\sum_{i=0}^{N-1}x^iy^{N-1-i}\right|_p|x-y|_p,\]
  we need only prove 
  \[\left|\sum_{i=0}^{N-1}x^iy^{N-1-i}\right|_p=|N\omega^{N-1}|_p.\]
  Noting that $x=\omega+x_0$ and $y=\omega+y_0$ for some $x_0,y_0\in D(0,r)$, we have 
  \[\left|\sum_{i=0}^{N-1}x^iy^{N-1-i}\right|_p=\left|\sum_{K=0}^{N-1}C_N^K\sum_{i+j=N-K-1}x_0^iy_0^j\omega^K\right|_p.\]
  Since 
  \[\left|C_N^K\sum_{i+j=N-K-1}x_0^iy_0^j\omega^K\right|_p
  \le |C_N^K|_pr^{N-K-1}|\omega|_p^K\le |C_N^K|_pp^{1-(N-K)}|\omega|_p^{N-1},\]
  by Lemma \ref{lem2.4}, we have 
  \[\left|C_N^K\sum_{i+j=N-K-1}x_0^iy_0^j\omega^K\right|_p
  < |N\omega^{N-1}|_p,\quad \forall\ 0\le K\le N-2.\]
  Since 
  \[\left|C_N^{K}\sum_{i+j=N-K-1}x_0^iy_0^j\omega^K\right|_p=|N\omega^{N-1}|_p,\quad K=N-1,\]
  by non-Archimedean triangle inequality we have 
  \[\left|\sum_{K=0}^{N-1}C_N^K\sum_{i+j=N-K-1}x_0^iy_0^j\omega^K\right|_p=|N\omega^{N-1}|_p.\] 
\end{proof}
\subsection{\textit{p-adic dynamical systems}}
First, we recall some dynamical terminology.\par
Let $X$ be a set, and $\phi: X\rightarrow X$ be a map. The couple $(X,\phi)$ is a dynamical system. A point $x_0\in X$ is called a \textit{$k$-periodic point} of $\phi$ if for some $k\ge 1$, $\phi^k(x_0)=x_0$ and $\phi^n(x_0)\neq x_0,\ \forall 1\le n\le k-1$. Especially, if $k=1$, we call $x_0$ a \textit{fixed point} of $\phi$.\par
Let $m\ge 1$, and $\mathcal{A}=\{0,1,\cdots,m-1\}$. We call the infinite product space $\Sigma_m=\mathcal{A}^{\mathbb{N}}$ the symbolic space of $m$ symbols. The left shift $\sigma$ on $\Sigma_m$ is defined as 
\[\sigma((x_0x_1x_2\cdots))=(x_1x_2x_3\cdots).\]
The couple $(\Sigma_m,\sigma)$ is called a \textit{full shift} on $m$ symbols. Let $A$ be an $m\times m$ matrix with entries in $\{0,1\}$ and let $\Sigma_A=\{x\in\Sigma_m: A_{x_ix_{i+1}}=1,\forall i\ge 0\}$. Then $(\Sigma_A,\sigma)$ forms a dynamical system called a \textit{subshift of finite type}.\par
Fan and Liao \cite{flw07} showed that a large family of $p$-adic dynamical systems are conjugate to full shift or subshift of finite type. Let $f:X\rightarrow \mathbb{Q}_p$ be a map with $X\subset f(X)$, where $X$ is a compact open set of $\mathbb{Q}_p$. For such $f$, we define 
\[\mathcal{K}_{f}:=\bigcap_{n=0}^{\infty}f^{-n}(X).\]
Suppose $X=\bigsqcup_{i\in I}D(c_i,p^{-\tau})$ where $I$ is a finite set. If there exists an integer $\tau_i\in \mathbb{Z}$ such that
\begin{align}\label{sfb}
  \forall x,y\in D(c_i,p^{-\tau}),\quad |f(x)-f(y)|_p=p^{\tau_i}|x-y|_p
\end{align}
and all $\tau_i$ in \eqref{sfb} are non-negative, but at least one is positive, then
we call $(\mathcal{K}_{f},f)$ a \textit{p-adic weak repeller}.\par
For any $i\in I$, let 
\[I_i:=\left\{j\in I:D(c_j,p^{-\tau})\subset f\left(D(c_i,p^{-\tau})\right)\right\}\]
and define the incidence matrix $A=(A_{i,j})_{I\times I}$ of $\mathcal{K}_{f}$, by $A_{i,j}=1$ if $j\in I_i$ and $A_{i,j}=0$ otherwise. The weak repeller $(K_f,f)$ is called transitive if $A$ is irreducible. The following Theorem \ref{thm2.6} shows that the chaotic dynamical behavior of a $p$-adic weak repeller is determined by $A$.
\begin{theorem}\cite[Theorem 1.1]{flw07}\label{thm2.6}
  Let $(\mathcal{K}_{f},f)$ be a transitive p-adic weak repeller in $\mathbb{Q}_p$ with incidence matrix $A$. Then $(\mathcal{K}_{f},f)$ is topologically conjugate to $(\Sigma_A,\sigma)$.
\end{theorem}
\subsection{\textit{p-adic polynomial dynamical systems}}
Let $f\in \mathbb{Z}_p[x]$ and $E\subset \mathbb{Z}_p$ be a compact $f$-invariant subset. Denote by $f_{(n)}$ the induced mapping of $f$ on $\mathbb{Z}/p^n\mathbb{Z}_p$, i.e.,
\[f_{(n)}(x\ {\rm mod}\ p^n)\equiv f(x)\quad {\rm mod}\ p^n.\]
We call $\sigma=(x_1,\cdots,x_k)\subset \mathbb{Z}/p^n\mathbb{Z}_p$ a cycle of $f_{(n)}$ of length $k$, if
\[f_{(n)}(x_1)=x_2,\cdots, f_{(n)}(x_i)=x_{i+1},\cdots, f_{(n)}(x_k)=x_1.\]
In this case we also say that $\sigma$ is a $k$-cycle at level $n$.\par
Let $g:= f^k$ be the $k$-th iterate of $f$. For all $x\in\sigma$, we denote
\begin{align}\label{an}
  \alpha_n(x):= g^{\prime}(x),
\end{align}
\begin{align}\label{bn}
  \beta_n(x):= \frac{g(x)-x}{p^n},
\end{align}
\begin{align}\label{An}
  A_n(x):= v_p(\alpha_n(x)-1),
\end{align}
\begin{align}\label{Bn}
  B_n(x):= v_p(\beta_n(x)).
\end{align}
Following \cite{dm01} and \cite{fl11}, we distinguish four cases.
\begin{enumerate}[(a)]
  \item If $\alpha_n\equiv 1\ ({\rm mod}\ p)$ and $\beta_n\not\equiv 0\ ({\rm mod}\ p)$, we say $\sigma$ \textit{grows}.
  \item If $\alpha_n\equiv 1\ ({\rm mod}\ p)$ and $\beta_n\equiv 0\ ({\rm mod}\ p)$, we say $\sigma$ \textit{splits}.
  \item If $\alpha_n\equiv 0\ ({\rm mod}\ p)$, we say $\sigma$ \textit{grows tails}.
  \item If $\alpha_n\not\equiv 0,1\ ({\rm mod}\ p)$, we say $\sigma$ \textit{partially splits}.
\end{enumerate}
Each $\sigma$ at level $n$ will be lift to several cycles at level $n+1$. We call them \textit{lifts} of $\sigma$ and denote any of them as $\tilde{\sigma}$.\par
The following results will play an important role in our proof.
\begin{proposition}\cite[Proposition 1]{fl11}\label{dt1}
  Let $n\ge 1$. Let $\sigma$ be a cycle of $f_{(n)}$ of length $k$ and $\tilde{\sigma}$ be a lift of $\sigma$. Then the following hold. 
  \begin{enumerate}[1)]
    \item If $\sigma$ grows or splits, then any lift $\tilde{\sigma}$ grows or splits.
    \item If $\sigma$ grows tails, then the single lift $\tilde{\sigma}$ grows tails.
    \item If $\sigma$ partially splits, then the lift $\tilde{\sigma}$ of the same length as $\sigma$ partially splits, and other lifts of length $kd$ grow or split, where $d$ is the order of $\alpha_n$ in $\left(\mathbb{Z}/p\mathbb{Z}\right)^*$.
  \end{enumerate}
\end{proposition}
\begin{proposition}\cite{dm01}\label{dt2}
  Let $\sigma$ be a growing cycle of $f_{(n)}$ and $\tilde{\sigma}$ be the unique lift of $\sigma$.
  \begin{enumerate}[1)]
    \item If $p\ge 3$ and $n\ge 2$, then $\tilde{\sigma}$ grows.
    \item If $p\ge 5$ and $n\ge 1$, then $\tilde{\sigma}$ grows.
    \item If $p = 3$ and $n=1$, then $\tilde{\sigma}$ grows if and only if $\beta_1(x)\not\equiv \frac{g^{\prime\prime}(x)}{2}\ ({\rm mod}\ p)$.
  \end{enumerate}
\end{proposition}
\begin{proposition}\cite{dm01}\label{dt4}
  Let $p\ge 3$ and $n\ge 1$. Let $\sigma$ be a partially splitting $k$-cycle of $f_{(n)}$ and $\tilde{\sigma}$ be a lift of $\sigma$ of length $kd$, where $d$ is the order of $\alpha_n$ in $\left(\mathbb{Z}/p\mathbb{Z}\right)^*$.
  \begin{enumerate}[1)]
    \item If $A_{n+1}<nd$, then $\tilde{\sigma}$ splits $A_{n+1}-1$ times then all lifts at level $n+A_{n+1}$ grow forever.
    \item If $A_{n+1}\ge nd$, then $\tilde{\sigma}$ splits at least $nd-1$ times.
  \end{enumerate}
\end{proposition}
The following theorems tell us how to link the property of $\sigma$ with the behavior of $p$-adic polynomial dynamical systems.
\begin{theorem}\cite{vsa06,jay09}\label{mini}
  Let $f\in \mathbb{Z}_p[x]$ and $E\subset \mathbb{Z}_p$ be a compact $f$-invariant set. Then, we have that $f:E\rightarrow E$ is minimal if and only if $f_{(n)}:E/p^n\mathbb{Z}_p\rightarrow E/p^n\mathbb{Z}_p$ is minimal for each $n\ge 1$.
\end{theorem}
\begin{theorem}\cite[p.2123]{fl11}\label{grta}
  If $\sigma=(x_1,\cdots,x_k)$ is a cycle of $f_{(n)}$ which grows tails, then $f$ admits a $k$-periodic point $x_0$ in the compact open set $X=\bigsqcup_{i=1}^k(x_i+p^n\mathbb{Z}_p)$ and $X$ is contained in the attracting basin of the periodic orbit $\left(x_0,f(x_0),\cdots,f^{k-1}(x_0)\right)$.
\end{theorem}
\begin{theorem}\cite[Theorem 1]{fl11}
  Let $f\in \mathbb{Z}_p[x]$ with deg$(f)\ge 2$. Then, we have the following decomposition
  \[\mathbb{Z}_p=\mathcal{P}\sqcup \mathcal{M}\sqcup \mathcal{B}\]
  where $\mathcal{P}$ is a finite set consisting of all periodic points of $f$, set $\mathcal{M}=\bigsqcup_i\mathcal{M}_i$ is the union of all (at most countably many) compact open invariant sets such that each $\mathcal{M}_i$ is a finite union of balls and each subsystem $f:\mathcal{M}_i\rightarrow \mathcal{M}_i$ is minimal, and each point in $\mathcal{B}$ lies in the attracting basin of a periodic orbit or of a minimal subsystem.
\end{theorem}
\section{Dynamical systems of $\left(\mathbb{P}^1(\mathbb{Q}_p),f_N\right)$}
Let $N=qp^m$ with $gcd(p,q)=1$. Let $a=p^{v_p(a)}(a_0+a_1p+\cdots)$ with $1\le a_0\le p-1$ and $0\le a_i\le p-1$ for all $i\ge 1$. In this section, we will prove Theorems \ref{sy1} -- \ref{sy4}. We distinguish three cases: $|a|_p>1,\ |a|_p<1$ and $|a|_p=1$.
\subsection{Case $|a|_p>1$}
We need the following lemma.
\begin{lemma}\label{lem3.1}
  Suppose $|a|_p>1$. If $|x|_p^{N-1}<|a|_p$, then there exists a positive integer $n_0$ such that 
  \[|f^{n_0}_N(x)|_p=\frac{1}{|a|_p}.\]
\end{lemma}
\begin{proof}
  We partition the ball $\left\{x\in \mathbb{Q}_p:|x|^{N-1}_p<|a|_p\right\}$ into
  \[A_1:=\left\{x\in \mathbb{Q}_p:|x|^{N-1}_p<1\right\}\]
  and
  \[A_2:=\left\{x\in \mathbb{Q}_p:1\le |x|^{N-1}_p<|a|_p\right\}.\]
  \par
  If $x\in A_1$, then $|f_N(x)|_p=\frac{1}{|a|_p}$.\par
  We claim that for all $x\in A_2$, there exists $n_0$ such that $f_N^{n_0}(x)\in A_1$. By contrary, suppose there exists an $x$ such that for all positive integers $n$, we have $f^n_N(x)\in A_2$.
  Since $x\in A_2$, we have $|f_N(x)|_p\le \frac{|x|_p^N}{|a|_p}< |x|_p$. Since $f_N(x)\in A_2$, we have $1\le |f_N(x)|_p<|x|_p$ which implies $|a|_p\ge p^N$. Since $f_N^2(x)\in A_2$, by the same arguments, we have $|a|_p\ge p^{2N-1}$. By repeating the arguments, we have $|a|_p\ge p^{n(N-1)+1}$ for all positive integers $n$. This leads to a contradiction.
\end{proof}
\begin{proposition}\label{pro3.2}
  Suppose $p\ge 3,\ |a|_p>1$. Then, 
  \[\lim_{n\to\infty}f_N^n(x)=x_0,\quad\forall x\in S\left(0,\frac{1}{|a|_p}\right)\] 
  where $x_0$ is the fixed point of $f_N$ with $|x_0|_p=\frac{1}{|a|_p}$.
\end{proposition}
\begin{proof}
  Let $H_N\in\mathbb{Z}_p[x]$ be defined by $H_N(y)=(\frac{y}{a})^N-y+1,\ \forall y\in \mathbb{Z}_p$. Then, for all $y\in\mathbb{Z}_p$,
  \[|H^{\prime}_N(y)|_p=\left|\frac{Ny^{N-1}}{a^N}-1\right|_p=1,\] 
  which implies $H^{\prime}_N(y)\not\equiv 0\ ({\rm mod}\ p)$.\par
   Since $H_N(y)\equiv 0\ ({\rm mod}\ p)$ only has one solution $y\equiv 1\ ({\rm mod}\ p)$, by Lemma \ref{lem2.2}, we have $H_N(y)=0$ has only one solution $y_0\in S(0,1)$. Let $x_0=\frac{y_0}{a}$. Then $|x_0|_p=\frac{1}{|a|_p}$ and $x_0$ is a fixed point of $f_N(x)$.\par
   Let $D_k:=D\left(k|a|_p,\frac{1}{p|a|_p}\right)$ with $1\le k\le p-1$. For all $x,y\in D_k$, by Corollary \ref{cor2.5}, we have
   \[|f_N(x)-f_N(y)|_p=\frac{1}{|a|_p}|N(k|a|_p)^{N-1}|_p|x-y|_p= \frac{|N|_p}{|a|^N_p}|x-y|_p.\]
   So, 
   \[f_N(D_k)=D\left(\frac{(k|a|_p)^N+1}{a},\ \frac{|N|_p}{p|a|_p^{N+1}}\right).\]
   Note that 
   \[S\left(0,\frac{1}{|a|_p}\right)=\bigsqcup_{k=1}^{p-1}D_k.\]
   Then, for any $1\le k\le p-1$, there exists exactly one $j$, such that $f_N(D_k)\subset D_j$. If $j_0 a|a|_p\equiv ax_0\ ({\rm mod}\ p) $, then $D_{j_0}=D\left(x_0,\frac{1}{p|a|_p}\right)$. Hence, $f_N(D_{j_0})\subset D_{j_0}$.
   Since
   \[\left|j|a|_p-\frac{(k|a|_p)^N+1}{a}\right|_p=\frac{1}{|a|_p}\left|ja|a|_p-1-(k|a|_p)^N\right|_p\le \frac{1}{|a|_p}\max\left\{\left|ja|a|_p-1\right|_p,\frac{1}{|a|^N_p}\right\},\]
   we see that such $j$ is independent of $k$, which implies $f_N(D_k)\subset D_{j_0}$ for all $k$. 
   Hence, $f_N\left(S\left(0,\frac{1}{|a|_p}\right)\right)\in D_{j_0}$. Since
   \[f_N^n(D_{j_0})=D\left(x_0,r^n\frac{1}{p|a|_p}\right),\]
   where $r=\frac{|N|_p}{|a|^N_p}<1$, we have $\lim_{n\to\infty}f_N^n(x)=x_0,\quad\forall x\in S\left(0,\frac{1}{|a|_p}\right)$.
\end{proof}
\begin{remark}
  If $p=2$, we define $D_{\pm 1}=D\left(\pm x_0,\frac{1}{4|a|_2}\right)$. Noting that $S\left(0,\frac{1}{|a|_2}\right)=D_{-1}\bigsqcup D_1$, we get the same conclusion as in Proposition \ref{pro3.2}.
\end{remark}
\begin{proof}[\textbf{Proof of Theorem \ref{sy1}}]
	Combining Proposition \ref{pro3.2} and Lemma \ref{lem3.1}, we need only consider two subcases.
 \begin{enumerate}[1)]
	 \item If $|x|_p^{N-1}>|a|_p$, then we have 
	 \[|f_N(x)|_p=\frac{|x|_p^N}{|a|_p}>|x|_p.\]
	 Hence, the absolute values of the iterations $f_N^n(x)$ are strictly increasing. Therefore,
	 \[\lim_{n\to\infty}f_N^n(x)=\infty,\quad\forall x\in \{\infty\}\cup\mathbb{Q}_p\setminus D\left(0,|a|_p^{\frac{1}{N-1}}\right).\]
	 \item If $|x|_p^{N-1}=|a|_p$, then by noting that 
	 \[\left|\frac{1}{a|a|_p}\right|_p=1,\quad \left|\frac{1}{a}|a|_p^{N-1}\right|_p=|a|_p^{-N}<1,\]
	 we have the polynomial $g_N$ defined by \eqref{dyn1} is of coefficients in $\mathbb{Z}_p$.
 \end{enumerate}
	 By defining $\pi(x)= |a|_p^{\frac{1}{N-1}}x$ as a conjugacy, we conclude that $\left(S\left(0,|a|_p^{\frac{1}{N-1}}\right), f_N\right)$ is topologically conjugate to $(S(0,1), g_N)$. 
\end{proof}
\subsection{Case $|a|_p<1$}
\begin{proposition}\label{pro3.6}
  Suppose $|a|_p<1$. Then, $\lim_{n\to\infty}f_N^n(x)=\infty,\quad\forall x\not\in S(0,1)$.
\end{proposition}
\begin{proof}
  For all $x\in\mathbb{Q}_p$ such that $|x|_p\ge p$, we have 
  \[|f_N(x)|_p=\frac{|x|_p^N}{|a|_p}>|x|_p.\]
  Thus, the absolute values of the iterations $f_N^n(x)$ are strictly increasing, and for any $x\in\mathbb{Q}_p$ with 
  $|x|_p\ge p$, we have 
  \[\lim_{n\to\infty}f_N^n(x)=\infty.\]
  On the other hand, for all $x\in\mathbb{Q}_p$ with $|x|_p\le \frac{1}{p}$, we have 
  \[|f_N(x)|_p=\frac{1}{|a|_p}\ge p.\]
  Then, repeating the above argument completes the proof.
\end{proof}
\begin{proposition}\label{pro3.7}
  Let $p\ge 3$. Suppose $|a|_p<1$ and $gcd(p-1,q)\nmid \frac{p-1}{2}$. Then, $\lim\limits_{n\to\infty}f_N^n(x)=\infty,\quad\forall x\in\mathbb{Q}_p\cup\{\infty\}$.
\end{proposition}
\begin{proof}
  By Lemma \ref{lem2.3}, if $gcd(p-1,q)\nmid \frac{p-1}{2}$, then $|x^N+1|_p\ge 1$ for all $x\in \mathbb{Q}_p$. Hence, for all $x\in\mathbb{Q}_p$,
  \[|f_N(x)|_p=\frac{|x^N+1|_p}{|a|_p}\ge \frac{1}{|a|_p}>1.\]
  Therefore, by Proposition \ref{pro3.6}, we have $\lim_{n\to\infty}f_N^n(x)=\infty,\quad\forall x\in \mathbb{Q}_p\cup\{\infty\}$.
\end{proof}
\begin{proposition}\label{pro3.8}
  Let $p\ge 3$. Suppose $|a|_p<1,\ gcd(p-1,q)\mid \frac{p-1}{2}$ and $|a|_p<|N|_p$. If $f_N$ has $\ell\ge 2$ different fixed points $\{\omega_{i,N}\}_{i=1}^{\ell}$ with $\omega_{i,N}\not\equiv \omega_{j,N}\ ({\rm mod}\ p)$ and $\ell\ge 2$, then $(\mathcal{K}_{f_N}, f_N)$ is topologically conjugate to $(\Sigma_{\ell}, \sigma)$, where 
  \[\mathcal{K}_{f_N}:=\bigcap_{n=0}^{\infty}f_N^{-n}(X)\quad \text{with}\quad X=\bigsqcup_{i=1}^{\ell} D(\omega_{i,N},\frac{1}{p}).\]
\end{proposition}
\begin{proof}
  Let $D_i=D(\omega_{i,N},\frac{1}{p})$. Then, by Corollary \ref{cor2.5}, for all $x,y \in D_i$, we have 
  \[|f_N(x)-f_N(y)|_p=\frac{|N|_p}{|a|_p}|x-y|_p>|x-y|_p.\]
  Hence, $f_N(D_i)=D\left(\omega_{i,N},\frac{|N|_p}{p|a|_p}\right)$.\par
  Noting that 
  \[|\omega_{i,N}-\omega_{j,N}|_p= 1\le \frac{|N|_p}{p|a|_p},\] 
  we have $D_j\subset f_N(D_i)$ for all $i,j\in \{1,\cdots ,\ell\}$. By Theorem \ref{thm2.6}, we complete the proof.
\end{proof}
\begin{corollary}\label{cor3.9}
  Let $p\ge 3$. Suppose $|a|_p<1,\ |a|_p<|N|_p^2$ and $gcd(p-1,q)\mid \frac{p-1}{2}$. If $gcd(p-1,q)\ge 2$, then $(\mathcal{A}_{f_N}, f_N)$ is topologically conjugate to $(\Sigma_{gcd(p-1,q)}, \sigma)$, where $\mathcal{A}_{f_N}=\left\{x\in\mathbb{Q}_p:|f_N^n(x)|_p=1,\ \forall n\ge 0\right\}$.
\end{corollary}
\begin{proof}
  By using Lemma \ref{lem2.2} in $\mathbb{Q}_p$, we have that if $\exists\ x_0\in \mathbb{Z}_p$ and $s,L\ge 1$ such that $F(x_0)\equiv 0\ ({\rm mod}\ p^s)$ and 
  \begin{align}\label{phl1}
    v_p\left(F^{(L)}(x_0)\right)-v_p(L!)<s-v_p\left(F^{\prime}(x_0)\right),
  \end{align}
  \begin{align}\label{phl2}
    v_p\left(F^{\prime}(x_0)\right)+v_p\left(F^{(k)}(x_0)\right)-v_p\left(F^{(k+1)}(x_0)\right)<s-v_p(k+1),\quad\forall\ 1\le k\le L-1,
  \end{align}
  (if $L=1$, we only need \eqref{phl1}) then there exists a unique $x\in \mathbb{Z}_p$ such that $F(x)=0$ and $x\equiv x_0\ ({\rm mod}\ p^{s-v_p(F^{\prime}(x_0))})$.\par
  Now, we will prove the number $\ell$ in Proposition \ref{pro3.8} is $gcd(p-1,q)$.
  By Lemma \ref{lem2.3}, the equation $x^N+1\equiv 0\ ({\rm mod}\ p)$ has $gcd(p-1,q)$ different solutions. Suppose $x_0$ is one of the solutions. Let $F(x)=x^N-ax+1$. Now, we distinguish two different cases.
  \begin{enumerate}[1)]
    \item If $gcd(N,p)=1$, then we choose $L=s=1$.
    Since $F(x_0)=x_0^N-ax_0+1\equiv x_0^N+1\equiv 0\ ({\rm mod}\ p)$ and $v_p(F^{\prime}(x_0))=v_p(N)=0$, we have \eqref{phl1} holds. Hence, there exists a unique $x\in \mathbb{Z}_p$ such that $F(x)=0$ and $x\equiv x_0\ ({\rm mod}\ p)$. So, we have $\ell=gcd(p-1,q)$.
    \item If $gcd(N,p)\neq 1$, then we choose $L=2$ and $s=2v_p(N)+1$.
    Since $|a|_p<|N|_p^2$, then we have 
    \[F(x)\equiv x^N+1\ \left({\rm mod}\ p^{2v_p(N)+1}\right).\]
    Let $x=\bar{x}+x_{v_p(N)}p^{v_p(N)}$, where $\bar{x}=x_0+x_1p+\cdots +x_{v_p(N)-1}p^{v_p(N)-1}$ with $0\le x_i\le p-1$ for all $1\le i\le v_p(N)$.
    Then, we have 
    \[F(x)\equiv \bar{x}^N+1+N\bar{x}^{N-1}x_{v_p(N)}p^{v_p(N)} \equiv \bar{x}^N+1+q\bar{x}^{N-1}x_{v_p(N)}p^{2v_p(N)}\ \left({\rm mod}\ p^{2v_p(N)+1}\right).\]
    Since $gcd(q\bar{x}^{N-1},p)=1$, if we have $\bar{x}^{N}+1\equiv 0\ \left({\rm mod}\ p^{2v_p(N)}\right)$, then there exists a unique $x_{v_p(N)}$ such that $F(x)\equiv 0\ \left({\rm mod}\ p^{2v_p(N)+1}\right)$.\par
    For equation $\bar{x}^{N}+1\equiv 0\ \left({\rm mod}\ p^{2v_p(N)}\right)$, repeat the arguments above. Since $x_0^N+1\equiv 0\ ({\rm mod}\ p)$, in the sense of modulo $p^{v_p(N)+1}$ congruence, we have a unique $\omega_0$ such that $F(\omega_0)\equiv 0\ \left({\rm mod}\ p^{2v_p(N)+1}\right)$.\par
    Since
    \[v_p(F^{\prime\prime}(\omega_0))-v_p(2!)=v_p(N)+v_p(N-1)=v_p(N)<v_p(N)+1=s-v_p(F^{\prime}(\omega_0)),\]
    and 
    \[2v_p(F^{\prime}(\omega_0))-v_p(F^{\prime\prime}(\omega_0))=v_p(N)<s-v_p(2),\]
    we obtain \eqref{phl1} and \eqref{phl2}. Then, there exists a unique $\omega$ with $\omega\equiv \omega_0\ \left({\rm mod}\ p^{v_p(N)+1}\right)$. Hence, we have $\ell=gcd(p-1,q)$.
  \end{enumerate}
  Finally, we prove $\mathcal{K}_{f_N}=\mathcal{A}_{f_N}$. It can be checked that $\mathcal{K}_{f_N}\subset\mathcal{A}_{f_N}$. On the other hand, if $x\not\in \mathcal{K}_{f_N}$, then there exists $n_0\in\mathbb{N}$ such that $f_N^{n_0}(x)\not\in X$. Hence, 
  \[\left(f_N^{n_0}(x)\right)^N+1\not\equiv 0\ ({\rm mod}\ p).\] 
  Thus, we have $\left|f_N^{n_0+1}(x)\right|_p\neq 1$, which implies $x\not\in\mathcal{A}_{f_N}$. Therefore, we have $\mathcal{A}_{f_N}\subset \mathcal{K}_{f_N}$.
\end{proof}
\begin{proposition}\label{pro3.10}
  Let $p\ge 3$. Suppose $|a|_p<1,\ gcd(p-1,q)\mid \frac{p-1}{2}$ and $|a|_p\ge |N|_p$. Then $\lim_{n\to\infty}f_N^n(x)=\infty,\quad\forall x\in\mathbb{Q}_p\cup\{\infty\}$.
\end{proposition}
\begin{proof}
  Let $D_j:= D(j,\frac{1}{p})$. Note that $S(0,1)=\bigsqcup_{j=1}^{p-1} D_j$.
  By Proposition \ref{pro3.6}, we need only consider $x\in S(0,1)$. By Corollary \ref{cor2.5}, for all $x,y\in D_j$ we have 
  \[|f_N(x)-f_N(y)|_p=\frac{|N|_p}{|a|_p}|x-y|_p\le|x-y|_p.\]
  So, 
  \[f_N(D_j)=D\left(\frac{j^N+1}{a},\ \frac{|N|_p}{p|a|_p}\right).\]
  If $|j^N+1|_p>|a|_p$, then for all $x\in f_N(D_j)$, we have $|x|_p>1$. If $|j^N+1|_p<|a|_p$, then for all $x\in f_N(D_j)$, we have $|x|_p<1$. Hence, if $|j^N+1|_p\neq |a|_p$, then for all $x\in f_N(D_j)$, we have $\lim_{n\to\infty}f_N^n(x)=\infty$.\par
  Now, we claim that there does not exist $j\in \{1,\cdots,p-1\}$, such that $|j^N+1|_p=|a|_p$. Conversely, if there exists $j_0\in \{1,\cdots,p-1\}$ such that $|j_0^N+1|_p=|a|_p<1$. Then, $j_0^N+1\equiv 0\ ({\rm mod}\ p)$. Let $g$ be a generator of $\mathbb{F}_p^*$. By checking the proof of Lemma \ref{lem2.3}, we have $j_0=g^y$, where $y$ satisfies 
  \[qy\equiv \frac{p-1}{2}\ ({\rm mod}\ p-1).\]
  Let $qy=\frac{p-1}{2}+(p-1)t$ with $t\in\mathbb{N}$. Then,
  \[j_0^N+1=g^{qp^my}+1\equiv \left(g^{\frac{p-1}{2}}\right)^{p^m(1+2t)}+1\ ({\rm mod}\ p^{m+1}).\]
  Let $g^{\frac{p-1}{2}}=-1+kp$ with $k\in \mathbb{Z}$. Then,
  \[\left(g^{\frac{p-1}{2}}\right)^{p^m}=\left(-1+kp\right)^{p^m}=-1+kp^{m+1}+\sum\limits_{t=2}^{p^m}C_{p^m}^t\cdot (kp)^t\cdot (-1)^{p^m-t}.\]
  By Lemma \ref{lem2.4}, we have 
  \[\left|C_{p^m}^t\cdot p^t\right|_p=\left|C_{p^m}^{p^m-t}\right|_p\cdot p^{-t}\le p^{-m-1},\quad \forall\ 2\le t\le p^m.\]
  Hence,
  \[\left(g^{\frac{p-1}{2}}\right)^{p^m(1+2t)}+1\equiv (-1)^{1+2t}+1\equiv 0\ ({\rm mod}\ p^{m+1}).\]
  Therefore, $|j^N+1|_p<|N|_p\le |a|_p$, which leads to a contradiction.
\end{proof}
\begin{proof}[\textbf{Proof of Theorem \ref{sy2}}]
  Combining Propositions \ref{pro3.7}, \ref{pro3.8} and \ref{pro3.10} and Corollary \ref{cor3.9}, we complete the proof.
\end{proof}
\begin{proposition}\label{pro3.11}
  If $|a|_2=2^{-k}$ with $k\ge 1$ and $2\nmid N$, then $(\mathbb{P}^1(\mathbb{Q}_2), f_N)$ has exactly two fixed points $\infty$ and $x_0\in D(2^k-1,2^{-k-1})$. Moreover, 
    \[\lim_{n\to\infty}f_N^n(x)=\infty,\quad \forall x\neq x_0.\] 
\end{proposition}
\begin{proof}
  Let $D_i=D(2^i-1,2^{-i-1})$ with $1\le i\le k$. Then,
  \[S(0,1)=D(-1,2^{-k-1})\sqcup\left(\bigsqcup_{i=1}^kD_i\right).\]
  By Lemma \ref{lem2.4}, we have
  \[|f_N(x)-f_N(y)|_2=2^k|x-y|_2,\]
  for all $x,y\in D(1,\frac{1}{4})$ or 
  $x,y\in D(-1,\frac{1}{4})$.
  Hence,
  \[f_N(D_i)=D\left(\frac{(2^i-1)^N+1}{a},2^{k-1-i}\right)\]
  and
  \[f_N(D(-1,2^{-k-1}))=D(0,2^{-1}).\]
  Since 
  \[(2^i-1)^N+1\equiv (-1)^N+(-1)^{N-1}2^iN+1\equiv 2^i\ ({\rm mod}\ 2^{i+1}),\]
  we have 
  \[\left|\frac{(2^i-1)^N+1}{a}\right|_2=2^{k-i},\]
  which implies that
  \[|f_N(x)|_2=2^{k-i},\quad \forall x\in D_i.\]
  So, if $x\not\in D_k$, then $f_N(x)\not\in S(0,1)$. By Proposition \ref{pro3.6}, we have 
  \[\lim_{n\to\infty}f_N^n(x)=\infty,\quad\forall x\not\in D_k.\]
  \par
  Now we will prove that $f_N$ have exactly two fixed points $\infty$ and $x_0\in D_k$. We claim that for any $t\ge 1$, we have that $x^N+1\equiv 0\ ({\rm mod}\ 2^t)$ has only one solution $x\equiv -1\ ({\rm mod}\ 2^t)$. We prove it by induction on $t$. Suppose $t=1$, this is clear. Suppoes $x^N+1\equiv 0\ ({\rm mod}\ 2^s)$ has only one solution $x\equiv -1\ ({\rm mod}\ 2^s)$. Let $x=-1+2^su$. Then,
  \[x^N+1\equiv (-1+2^su)^N+1\equiv 2^su\ ({\rm mod}\ 2^{s+1}).\]
  Hence, 
  \[x^N+1\equiv 0\ ({\rm mod}\ 2^{s+1})\]
  has solution $t\equiv 0\ ({\rm mod}\ 2)$. Therefore, we have that $x^N+1\equiv 0\ ({\rm mod}\ 2^{s+1})$ has only one solution $x\equiv -1\ ({\rm mod}\ 2^{s+1})$. Let $H_N(x)=x^N-ax+1$. Since $H_N(x)\equiv x^N+1\ ({\rm mod}\ 2^k)$, we have that $H_N(x)\equiv 0\ ({\rm mod}\ 2^k)$ has only one solution $x\equiv -1\ ({\rm mod}\ 2^{k})$. So, by the same conclusion above, we have that $H_N(x)\equiv 0\ ({\rm mod}\ 2^{k+1})$ has only one solution $x\equiv 2^k-1\ ({\rm mod}\ 2^{k+1})$. Since $H_N^{\prime}(\alpha)\equiv 1\ ({\rm mod}\ 2)$ for any $\alpha\equiv 2^k-1\ ({\rm mod}\ 2^{k+1})$, we have 
  \[|H_N(\alpha)|_2\le 2^{-k-1}<1=|H_N^{\prime}(\alpha)|_2^2.\]
  So, by Lemma \ref{lem2.2}, the equation $H_N(x)=0$ has only one solution $x_0\in D_k$, which implies that $f_N$ have exactly two fixed points $\infty$ and $x_0\in D_k$. \par
  Finally, we prove that
  \[\lim_{n\to\infty}f_N^n(x)=\infty,\quad \forall x\in D_k\setminus \{x_0\}.\]
  Since $D_k=D(x_0,2^{-k-1})$, we have 
  \[|f_N(x)-x_0|_2=|f_N(x)-f_N(x_0)|_2=2^k|x-x_0|_2,\quad\forall x\in D_k\setminus \{x_0\}.\]
  Hence, for any $x\in D_k\setminus \{x_0\}$, there exists $n_0\ge 1$, such that $f_N^{n_0}(x)\not\in D(x_0,2^{-k-1})=D_k$. Therefore, 
  \[\lim_{n\to\infty}f_N^n(x)=\infty,\quad \forall x\in D_k\setminus \{x_0\}.\]
\end{proof}
\begin{proposition}\label{pro3.12}
  Suppose $|a|_2=\frac{1}{2}$ and $|N|_2\le\frac{1}{4}$. Then, $f_N$ has exactly two fixed points $\infty$ and $x_0\in S(0,1)$. Moreover,
  \[\lim_{n\to\infty}f_N^n(x)=x_0,\quad\forall x\in S(0,1).\]
\end{proposition}
\begin{proof}
  Let $H_N\in\mathbb{Z}_p[x]$ be defined by $H_N(x)=x^N-ax+1,\ \forall x\in \mathbb{Z}_p$.
  We distinguish two subcases.
  \begin{enumerate}[1)]
    \item Suppose $a\equiv 2\ ({\rm mod}\ 8)$. Then, the equation $H_N(x)\equiv 0\ ({\rm mod}\ 8)$ has solution $\alpha\equiv 1\ ({\rm mod}\ 4)$. So, 
    \[|H_N(\alpha)|_2\le \frac{1}{8}< \frac{1}{4}=|H^{\prime}_N(\alpha)|_2^2.\]
    By Lemma \ref{lem2.2}, the equation $H_N(x)= 0$ has only one solution $x_0\in D(1,\frac{1}{4})$. Therefore, $f_N$ has exactly two fixed points $\infty$ and $x_0\in D(1,\frac{1}{4})$.
    \item Suppose $a\equiv 6\ ({\rm mod}\ 8)$. Then, by the same arguments above, we have that $f_N$ has exactly two fixed points $\infty$ and $x_0\in D(-1,\frac{1}{4})$.
  \end{enumerate}
  Since $S(0,1)=D(1,\frac{1}{4})\sqcup D(-1,\frac{1}{4})$ and 
    \[|f_N(x)-f_N(y)|_2=2|N|_2|x-y|_2 \le \frac{1}{2}|x-y|_2,\ \forall x,y \in D(\pm 1,\frac{1}{4}),\]
    we have $f_N(D(\pm 1,\frac{1}{4}))= D\left(\frac{2}{a},\frac{|N|_2}{2}\right)$,
  which implies our conclusion.
\end{proof}
\begin{proof}[\textbf{Proof of Theorem \ref{sy3}}]
  Combining Propositions \ref{pro3.6}, \ref{pro3.11} and \ref{pro3.12}, we need only consider two subcases with $x\in S(0,1)$.
  \begin{enumerate}[1)]
    \item Suppose $|a|_2\le \frac{1}{4}$ and $2\mid N$. For all $x\in D\left(\frac{2}{a},\frac{1}{4|a|_2}\right)$, we have $|x|_2\neq 1$. Noting that $f_N(D(\pm 1,\frac{1}{4}))= D(\frac{2}{a},\frac{1}{4})$, we have $\lim_{n\to\infty}f_N^n(x)=\infty,\quad\forall x\in S(0,1)$.
    \item Suppose $|a|_2= \frac{1}{2}$ and $|N|_2=\frac{1}{2}$. By Lemma \ref{lem2.4}, we have
    \[\left|\frac{C_N^{N-k}2^k}{2a}\right|_2\le 2|N|_2=1,\quad \forall 2\le k\le N.\]
    Since 
    \[\frac{|N|_2}{|a|_2}=1\ \text{and}\ \frac{|2-a|_2}{|2a|_2}\le 1,\]
    we have that the polynomial $h_N$ defined by \eqref{hn} is of coefficients in $\mathbb{Z}_p$.
  \end{enumerate} 
  \par
    By defining $\pi(x)=\frac{x-1}{2}$ as a conjugacy, we conclude that the subsystem $\left(S(0,1), f_N\right)$ is topologically conjugate to $(\mathbb{Z}_2, h_N)$.
\end{proof}
\subsection{Case $|a|_p=1$}
\begin{proof}[\textbf{Proof of Theorem \ref{sy4}}]
  The proof is similar to that of Proposition \ref{pro3.6}. 
\end{proof}
\section{Three $p$-adic polynomial dynamical systems}
First, we discuss the $p$-adic polynomial dynamical system $(S(0,1),g_N)$.
\begin{proposition}\label{pro4.1}
  Suppose $|a|_p>1$ and $(N-1)\mid v_p(a)$. Let $x\in\mathbb{Q}_p$, such that $|x|_p^{N-1}=|a|_p=p^{(N-1)s}$ and $D_j:=D(jp^{-s},p^{s-1}),\ 1\le j\le p-1$. If $x\in D_{j_0}$, then $f_N(x)\in D_{j_1}$ with $a_0j_1\equiv j_0^N\ ({\rm mod}\ p)$.
\end{proposition}
\begin{proof}
  If $p=2$, then the proposition holds obviously. If $p\ge 3$, then by Corollary \ref{cor2.5}, for all $x,y \in D_{j_0}$, we have
  \[|f_N(x)-f_N(y)|_p=|N|_p|x-y|_p\le |x-y|_p.\]
  Hence,
  \[f_N(D_{j_0})=D\left(\frac{(j_0p^{-s})^N+1}{a},\ p^{s-1}|N|_p\right)\subset S(0,p^s).\]
  Since $|N|_p\le 1$ and $S(0,p^s)=\bigsqcup_{j=1}^{p-1}D_j$, there exists exactly one $j_1$ with $f_N(D_{j_0})\subset D_{j_1}$.\\
  Since
  \[\left|j_1p^{-s}-\frac{(j_0p^{-s})^N+1}{a}\right|_p=
  |(j_0p^{-s})^N-j_1ap^{-s}+1|_p\cdot p^{(1-N)s}\le p^{s-1},\]
  we have
  \[p^{Ns}\left((j_0p^{-s})^N-j_1ap^{-s}+1\right) \equiv 0\ ({\rm mod}\ p).\]
  A simple computation gives $a_0j_1\equiv j_0^N\ ({\rm mod}\ p)$ as required.
\end{proof}
\begin{proposition}\label{pro4.2}
  Suppose $|a|_2>1$ and $(N-1)\mid v_2(a)$. Define $g_N$ as \eqref{dyn1}. Then, the whole dynamical system $(S(0,1),g_N)$ can never be minimal.
\end{proposition}
\begin{proof}
  Conversely, if $(S(0,1),g_N)$ is minimal, then by Theorem \ref{mini}, we have that $\sigma$ is a cycle of length $2^{n-1}$ at level $n$ for each $n\ge 1$.
  Hence, we have 
  \begin{align}\label{ty1}
    \alpha_n(1)\equiv 1\ ({\rm mod}\ 2)\quad \text{and}\quad \beta_n(1)\not\equiv 0\ ({\rm mod}\ 2).
  \end{align}
  Let $|a|_2=2^{(N-1)s}$ with some $s\ge 1$.
  Since $\sigma$ is a cycle of length $1$ at level $1$, we have
  \[|\alpha_1(1)-1|_2=|N-2^{(N-1)s}a|_2,\]
  \[|\beta_1(1)|_2=2\times|1-2^{(N-1)s}a+2^{Ns}|_2.\]
  Since \eqref{ty1} holds for $n=1$, we have $N\equiv 1\ ({\rm mod}\ 2)$ and $a_1=1$.\par
  Since $\sigma$ is a cycle of length $2$ at level $2$, we have
  \[|\beta_2(1)|_2=4\times|(1+2^{Ns})^N+2^{N^2s}a^N-2^{(N^2-1)s}a^{N+1}|_2.\]
  Since we have $N\equiv 1\ ({\rm mod}\ 2)$ and $a_1=1$, it can be checked that
  \[(1+2^{Ns})^N+2^{N^2s}a^N-2^{(N^2-1)s}a^{N+1}\equiv 0\ ({\rm mod}\ 8),\]
  which implies $\beta_2(1)\equiv 0\ ({\rm mod}\ 2).$ This leads to a contradiction.
\end{proof}
Now, we provide a necessary and sufficient condition under which $\sigma$ is a cycle of length $p-1$ at level $1$.
\begin{proposition}\label{pro4.3}
  Let $p\ge 3$ and let $g$ be a generator of $\mathbb{F}_p^*$. Then, $\sigma$ is a cycle of length $p-1$ at level $1$ if and only if \eqref{bcy} holds.
\end{proposition}  
\begin{proof}
  Recall that $\pi: x\mapsto |a|_p^{N-1}x$ is a conjugacy from $\left(S\left(0,|a|_p^{\frac{1}{N-1}}\right), f_N\right)$ to $(S(0,1), g_N)$. Define $D_j$ as in the Proposition \ref{pro4.1}. Since $\pi(D_j)= j+p\mathbb{Z}_p$, by combining Theorem \ref{sy1} and Proposition \ref{pro4.1}, we have that $\sigma$ is a cycle of length $p-1$ at level $1$ if and only if
  \begin{align}\label{43}
    \min\left\{n\ge 1: f_N^n(D_1)\subset D_1\right\}=p-1.
  \end{align} 
  Since $a_0g^{m}\equiv 1\ ({\rm mod}\ p)$, we deduce $j_1\equiv g^mj_0^N\ ({\rm mod}\ p)$ from $a_0j_1\equiv j_0^N\ ({\rm mod}\ p)$.
  Hence,
  \[j_n\equiv g^{\frac{N^n-1}{N-1}m}j_0^{N^n}\ ({\rm mod}\ p).\]
  Since $g$ is a generator of $\mathbb{F}_p^*$, we have \eqref{43} holds if and only if \eqref{bcy} holds.
\end{proof}
\begin{proof}[\textbf{Proof of Theorem \ref{dsy1}}]
  The case $p=2$ has been proven in Proposition \ref{pro4.2}. By the same arguments in the proof of Proposition \ref{pro4.2}, we have that the whole system is minimal if and only if $\sigma$ is a cycle of length $(p-1)p^{n-1}$ at level $n$ for each $n\ge 1$. By Proposition \ref{pro4.3}, we have that $\sigma$ is a cycle of length $p-1$ at level $1$ if and only if \eqref{bcy} holds. By Proposition \ref{dt2}, we need only check that if $p=3$, then \eqref{bcy} holds if and only if $N\equiv 1\ ({\rm mod}\ 2)$ and $a_0=2$.
  \par
  If \eqref{bcy} holds, then $m\equiv 1\ ({\rm mod}\ 2)$ and $(N+1)\cdot m \equiv 0\ ({\rm mod}\ 2)$.
  Hence,
  $N\equiv 1\ ({\rm mod}\ 2)$ and $2^1a_0\equiv 1\ ({\rm mod}\ 3)$.
  The only if part follows by elementary computation. 
\end{proof}
Now, we turn to consider the dynamical system $(\mathbb{Z}_p,f_N)$ with $|a|_p=1$. \par
\begin{proof}[\textbf{Proof of Theorem \ref{dsy2}}]
  Since 
  $|f_N(0)-1|_2=|a-1|_2\le \frac{1}{2}$ and $|f_N(1)|_2=\frac{1}{2}$,
  we have that $\sigma$ is a cycle of length $2$ at level $1$.\par
  Noting that $\alpha_1(0)=0$, we have $\sigma$ grows tails at level $1$. By Theorem \ref{grta}, we complete the proof.
\end{proof}
\begin{proposition}\label{pro4.4}
  Suppose $a\equiv 1\ ({\rm mod}\ 3)$. Then, we have $f_N(3\mathbb{Z}_3)\subset f_N(1+3\mathbb{Z}_3)$ and $ f_N(1+3\mathbb{Z}_3)\subset f_N(2+3\mathbb{Z}_3)$. Moreover, the compact open set $2+3\mathbb{Z}_3$ is an $f_N$-invariant set if and only if $2\mid N$.
\end{proposition}
\begin{proof}
  Since $|f_N(0)-1|_3=|a-1|_3$ and $|f_N(1)-2|_3=|2a-2|_3$, we have $f_N(3\mathbb{Z}_3)\subset f_N(1+3\mathbb{Z}_3)$ and $ f_N(1+3\mathbb{Z}_3)\subset f_N(2+3\mathbb{Z}_3)$.\par
  Since $|f_N(2)-2|_3=|2^N+1-2a|_3$ and $2^N+1-2a\equiv (-1)^N-1\ ({\rm mod}\ 3)$, we have $2+3\mathbb{Z}_3$ is an $f_N$-invariant set if and only if $2\mid N$.
\end{proof}
\begin{proposition}\label{pro4.5}
  Suppose $a\equiv 1\ ({\rm mod}\ 3)$ and $6\mid N$. Then $f_N$ admits a fixed point $x_0$ in $2+3\mathbb{Z}_3$, and $2+3\mathbb{Z}_3$ is contained in the attracting basin of $x_0$. 
\end{proposition}
\begin{proof}
  By Proposition \ref{pro4.4}, we have $\sigma$ is a cycle of length $1$ at level $1$. Since $|\alpha_1(2)|_3=|N\cdot 2^{N-1}|_3\le \frac{1}{3}$, we have $\alpha_1(2)\equiv 0\ ({\rm mod}\ 3)$. So $\sigma$ grows tails at level $1$. By Theorem \ref{grta}, we complete the proof.
\end{proof}
\begin{proposition}\label{pro4.6}
  Suppose $a\equiv 1,4\ ({\rm mod}\ 9)$ and $N\equiv 2\ ({\rm mod}\ 6)$. Then, $\sigma$ grows at level $1$. Moreover, if $a\equiv 4\ ({\rm mod}\ 9)$, then $(2+3\mathbb{Z}_3)$ is a minimal set.
\end{proposition}
\begin{proof}
  By the same arguments in the proof of Proposition \ref{pro4.5}, we have $\sigma$ grows at level $1$.
  Since 
  \[2^N+1-2a\equiv 5-2a\ ({\rm mod}\ 9)\]
  and $|\beta_1(2)|_3=3\times |2^N+1-2a|_3$, we have $\beta_1(2)\equiv 0\ ({\rm mod}\ 3)$ if and only if $a\equiv 7\ ({\rm mod}\ 9)$.
  Since 
  \[2^{N+1}+2-4a-3N(N-1)\cdot 2^{N-2}\equiv -4a+4\ ({\rm mod}\ 9)\]
  and
  \[\left|\beta_1(2)-\frac{f^{\prime\prime}_N(2)}{2}\right|_3
  =3\times|2^{N+1}+2-4a-3N(N-1)\cdot 2^{N-2}|_3,\]
  we have $\beta_1(2)\equiv\frac{f^{\prime\prime}_N(2)}{2}\ ({\rm mod}\ 3)$ if and only if $a\equiv 1\ ({\rm mod}\ 9)$.
  By Theorems \ref{dt2} and \ref{mini}, we complete the proof.
\end{proof}
\begin{proposition}\label{pro4.7}
  Suppose $a\equiv 7\ ({\rm mod}\ 9)$ and $N\equiv 2\ ({\rm mod}\ 6)$. Then $\sigma$ splits at level $1$. Moreover, the compact open sets $(2+9\mathbb{Z}_3),\ (5+9\mathbb{Z}_3)$ and $(8+9\mathbb{Z}_3)$ are minimal sets if and only if $a\equiv 7\ ({\rm mod}\ 27)$ and $N\equiv 2\ ({\rm mod}\ 18)$, or $a\equiv 16\ ({\rm mod}\ 27)$ and $N\equiv 8,14\ ({\rm mod}\ 18)$.
\end{proposition}
\begin{proof}
  By Proposition \ref{pro4.6}, we have $\sigma$ splits at level $1$. Let $N=6s+2$ and $a=9k+7$, where $s,k$ are non-negative integers. 
  Since
  $|\beta_2(2)|_3=9\times|2^N+1-2a|_3$ and
  \[2^N+1-2a\equiv 4\cdot 2^{6s}+9k-13\ ({\rm mod}\ 27),\]
  we have $\beta_2(2)\equiv 0\ ({\rm mod}\ 3)$ if and only if $a\equiv 7\ ({\rm mod}\ 27),\ N\equiv 8\ ({\rm mod}\ 18)$, or $a\equiv 16\ ({\rm mod}\ 27),\ N\equiv 2\ ({\rm mod}\ 18)$, or $a\equiv 25\ ({\rm mod}\ 27),\ N\equiv 14\ ({\rm mod}\ 18)$.\par 
  By repeating the same computation to $\beta_2(5)$ and $\beta_2(8)$, we have $\beta_2(5)\equiv 0\ ({\rm mod}\ 3)$ if and only if $a\equiv 7\ ({\rm mod}\ 27),\ N\equiv 14\ ({\rm mod}\ 18)$, or $a\equiv 16\ ({\rm mod}\ 27),\ N\equiv 2\ ({\rm mod}\ 18)$, or $a\equiv 25\ ({\rm mod}\ 27),\ N\equiv 8\ ({\rm mod}\ 18)$. We also have $\beta_2(8)\equiv 0\ ({\rm mod}\ 3)$ if and only if $a\equiv 25\ ({\rm mod}\ 27),\ N\equiv 2\ ({\rm mod}\ 6)$.\par
  So, we have $\beta_2(x)\not\equiv 0\ ({\rm mod}\ 3)$ holds for $x=2,5,8$ if and only if $a\equiv 7\ ({\rm mod}\ 27)$ and $N\equiv 2\ ({\rm mod}\ 18)$, or $a\equiv 16\ ({\rm mod}\ 27)$ and $N\equiv 8,14\ ({\rm mod}\ 18)$. By Theorems \ref{dt2} and \ref{mini}, we complete the proof.
\end{proof}
\begin{proof}[\textbf{Proof of Theorem \ref{dsy3}}]
  Combining Propositions \ref{pro4.4}, \ref{pro4.5}, \ref{pro4.6} and \ref{pro4.7}, we need only prove the case $(\text{i})$ and the subcases $2),4),5)$ of $(\text{ii})$.
  \par
  First, we consider case $(\text{i})$. Note that if $2\nmid N$, we have $|f_N(2)|_3=|2^N+1|_3\le \frac{1}{3}$. By Proposition $\ref{pro4.4}$, we have $\sigma$ is a cycle of length $3$ at level $1$. By the same arguments in the proof of Theorem \ref{dsy2}, we complete the proof of this case.
  \par
  Since the proof of subcases $2),4),5)$ of $(\text{ii})$ are the same, we only prove the subcase $2)$ of $(\text{ii})$.\par
  Since $|\alpha_1(2)-2|_3=|N\cdot 2^{N-1}-2a|_3$, we have $\alpha_1(2)\equiv 2\ ({\rm mod}\ 3)$ if and only if $N\equiv 4\ ({\rm mod}\ 6)$. So, $\sigma$ partially splits if and only if $N\equiv 4\ ({\rm mod}\ 6)$. By elementary computation, we also have $\sigma=(2,8)$ at level $2$ if and only if $a\equiv 1\ ({\rm mod}\ 9)$. Let $N=6s+4$. Since
  \[|\alpha_2(2)-1|_3=|N^2\cdot 2^{N-1}(2^N+1)^{N-1}-a^{N+1}|_3\] 
  and
  \[N^2\cdot 2^{N-1}(2^N+1)^{N-1}-a^{N+1}\equiv 3s-3\ ({\rm mod}\ 9),\]
  we have $A_2=1$ if and only if $N\equiv 4,16\ ({\rm mod}\ 18)$.
  For $d=2$, we have $A_2=1<2$. By Proposition \ref{dt4}, we have $\sigma=(2,8)$ grows forever. By Theorem \ref{mini}, we complete the proof. 
\end{proof}
\textbf{Acknowledgements}: We are deeply grateful to Professor Lingmin Liao for advices that both influenced the course of this research and improved its presentation.

%
%
%
%

 { }


\vspace*{10ex}

\noindent Cheng Liu: School of Mathematics and Statistics,
Wuhan University, 
\vspace{-2mm}

\noindent\phantom{Cheng Liu: }Bayi Road 299, Wuchang District, Wuhan, China

\noindent\phantom{Cheng Liu: }e-mail: 2023202010021@whu.edu.cn 

\end{document}